\documentclass[11pt,leqno]{amsart}

\usepackage{tikz}
\usetikzlibrary{calc}
\usetikzlibrary{graphs}

\usepackage{amssymb}
\usepackage{amscd}
\usepackage{url} 

\newcommand{\ZZ}{{\mathbb Z}}

\DeclareMathOperator{\Hom}{Hom}

\newcommand{\Aut}{{\rm Aut}}

\newcommand{\supp}{{\rm supp}}

\newcommand{\Div}{{\rm Div}}
\newcommand{\Prin}{{\rm Prin}}

\usepackage{color} 
\usepackage{ulem}

\title{Galois points for a finite graph}


\author[Fukasawa]{Satoru Fukasawa*}
\thanks{*Corresponding author}
\address{Faculty of Science, Yamagata University, Kojirakawa-machi 1-4-12, Yamagata 990-8560, Japan} 
\email{s.fukasawa@sci.kj.yamagata-u.ac.jp} 

\author[Miezaki]{Tsuyoshi Miezaki}
\address{Faculty of Science and Engineering, Waseda University, Ohkubo 3-4-1, Shinjuku, Tokyo 169-8555, Japan} 
\email{miezaki@waseda.jp}

\newtheorem{theorem}{Theorem}[section]
\newtheorem{proposition}[theorem]{Proposition}
\newtheorem{corollary}[theorem]{Corollary}

\newtheorem{lemma}[theorem]{Lemma}

\theoremstyle{definition}
\newtheorem{definition}[theorem]{Definition} 
\newtheorem{remark}[theorem]{Remark}

\begin{document}
\begin{abstract} 
This paper introduces the notion of a Galois point for a finite graph, using the theory of linear systems of divisors for graphs discovered by Baker and Norine.   
We present a new characterization of complete graphs in terms of Galois points. 
\end{abstract}

\maketitle 

\section{Introduction} 
All graphs in this paper are finite, undirected, and simple.
Baker and Norine \cite{baker-norine} introduced linear systems of divisors for finite graphs as an analogue of linear systems of divisors for algebraic curves, to prove the Riemann--Roch theorem for graphs. 
The theory of linear systems for graphs has been developed by several authors (for example, see \cite{baker-norine, baker-norine2, corry, corry2, corry-perkinson}).  
In algebraic geometry, Yoshihara introduced the notion of a Galois point for a plane curve: a smooth point $P$ of a plane curve $C$ is called a Galois point if the covering map $\hat{C} \rightarrow \mathbb{P}^1$ induced by the projection from $P$ is a Galois covering, where $\hat{C}$ is the smooth model of $C$.   
Several classification results of algebraic varieties have been obtained, according to the arrangement of Galois points
(for example, see \cite{fukasawa, miura-yoshihara, yoshihara, open}). 

One purpose of this paper is to introduce the notion of a {\it Galois point for a finite graph}. 
To define a Galois point, the notion of a Galois covering is required. 
The notion of a harmonic group action was introduced by Corry (see \cite[Definition 2.4]{corry}, \cite[Definition 2.5]{corry2}). 
Using this notion and referring to \cite[Chapter IV, Proposition 3.1]{hartshorne} and \cite[Theorem 3.7.1, Corollary 3.7.2, Theorem 3.8.2]{stichtenoth}, we present the following definition of a Galois point. 

\begin{definition} \label{Galois point}
Let $G$ be a 
2-edge-connected graph, 
and let $D$ be a divisor with $r(D)=2$. 
A vertex $P \in V(G)$ is called a {\it Galois point} with respect to $|D|$ if the following three conditions are satisfied: 
\begin{itemize}
\item[(1)] $r(D-P)=1$, 
\item[(2)] for any vertex $Q \in V(G)$ (including the case $Q=P$), $r(D-P-Q)=0$,
\item[(3)] there exist a subgroup $H \subset {\rm Aut}(G)$ 
with $|H|=\deg (D)-1$ and different effective divisors $E_1,E_2 \in |D-P|$ such that
\begin{enumerate}
\item 
$|V(G/H)|>1$, 
\item 
$H$ acts harmonically on $G$, 
\item
$\sigma(E_1)=E_1, \ \sigma(E_2)=E_2$ 
for any $\sigma \in H$. 
\end{enumerate}
\end{itemize} 
\end{definition}

\begin{remark}
The conditions in Definition \ref{Galois point} are explained in algebraic geometry, for a smooth projective curve $C$ with a divisor $D$, as follows.  
The assumption $r(D)=2$ implies that the linear system $|D|$ induces a rational map $\varphi_{|D|}: C \dashrightarrow \mathbb{P}^2$. 
Condition (1) implies that $P$ is not a base point of $|D|$, namely, $\varphi_{|D|}$ is defined at $P$, and the linear system $|D-P|$ induces a rational map $\varphi_{|D-P|}: C \dashrightarrow \mathbb{P}^1$.
Condition (3) implies that $\varphi_{|D-P|}$ coincides with the quotient map $C \rightarrow C/H$, namely, $|D-P|$ is base-point-free.  
In conclusion, $\deg \varphi_{|D|}(C)=\deg D$, $\varphi_{|D|}(P)$ is a smooth point of a plane curve $\varphi_{|D|}(C)$, the projection from $\varphi_{|D|}(P)$ is identified with $\varphi_{|D-P|}$, and $\varphi_{|D|}(P)$ is a Galois point of $\varphi_{|D|}(C)$. 

In algebraic geometry, condition (2) is not required for the definition of a Galois point, since it can be derived from conditions (1) and (3). 
Conditions (1) and (2) suggest a definition of a {\it smooth point} of a graph as a plane curve. 
\end{remark}

Another purpose of this paper is to use the notion of a Galois point to characterize complete graphs. 
The following is the primary theorem around which this paper is centered. 

\begin{theorem} \label{graph}
Let $G$ be a 
2-edge-connected graph, 
and let $D=P_1+\cdots+P_n$, 
where $\{P_1, \ldots, P_n\}=V(G)$. 
Then the following conditions are equivalent: 
\begin{itemize}
\item[(1)] the graph $G$ coincides with the complete graph $K_n$, 
\item[(2)] $r(D)=2$, and there exist $i, j$ with $i \ne j$ such that vertices $P_i, P_j$ are Galois points with respect to $|D|$. 
\end{itemize} 
In this case, all vertices $P_1, \ldots, P_n \in V(G)$ are Galois points. 
\end{theorem} 

This paper is organized as follows. 
In Section~\ref{sec:pre}, 
we define, and give some basic properties of, 
linear systems and harmonic actions on the graphs used in this paper.
In Section~\ref{sec:ex}, 
we provide examples of Galois points of certain graphs. 
In Section~\ref{sec:proof}, 
we prove Theorem~\ref{graph}.
Finally, in Section~\ref{sec:rem}, 
we provide concluding remarks about future work. 


\section{Preliminaries}\label{sec:pre}

\subsection{Divisors on finite graphs}

In this paper, we always assume that 
a graph $G$ is finite, undirected, 
and connected without loops or multiple edges. 
The set of vertices (resp.~edges) is denoted by $V(G)$ (resp.~by $E(G)$). 
For $e=\{P,Q\} \in E(G)$, we write $\overline{PQ}$. 
A {\it divisor} on a graph $G$ is an element of the free abelian group on $V(G)$: 
\[
\Div(G)=\left.\left\{\sum_{P\in V}a_P P\ \right\vert\ a_P\in \ZZ\right\}. 
\]
For $D=\sum_{P\in V(G)} a_P P\in \Div(G)$, 
let 
\[\deg(D):=\sum_{P\in V(G)} a_P,\ D(P):=a_P,\] 
\[
\supp(D):=\{P\in V\mid D(P)\neq 0\}, 
\]
and 
\[
\Div^k(G):=\{D\in \Div(G)\mid \deg(D)=k\}. 
\]
For $D,D'\in \Div(G)$, 
$D \geq D'$ if and only if 
$D(P) \geq D'(P)$ for all $P\in V(G)$. 
A divisor $D \in \Div(G)$ is called {\it effective} if $D \geq 0$. 

Let $M(G) = \Hom(V(G),\ZZ)$ be 
the set of all integer-valued functions 
on $V(G)$. 
The {\it Laplacian operator} $\Delta: M(G)\rightarrow \Div(G)$ is defined as follows. 
For $f\in M(G)$, 
\[
\Delta(f):=\sum_{P\in V(G)}\sum_{\overline{PQ}\in E(G)}(f(P)-f(Q))P. 
\]
We note that the {\it principal divisor} 
$\Prin(G):=\Delta(M(G))$ is a subgroup of $\Div^0(G)$. 
For $D,D'\in \Div(G)$, $D\sim D'$ if and only if 
\[
D-D'\in \Prin(G). 
\]
We call this relation a {\it linear equivalence}. 
Let $D$ be a divisor. 
Then the {\it (complete) linear system} $|D|$ is the set of all effective divisors linearly equivalent to $D$, 
\[
|D|:=\{ E \in  \Div(G)\mid E\geq  0, E \sim D \}. 
\]
The smallest integer $s$ such that $|D-E| \ne \emptyset$ for any effective divisor $E$ with $\deg (E)=s$ is denoted by $r(D)$. 
If $|D|=\emptyset$, then we set $r(D)=-1$ and 
call $r(D)$ the {\it rank} of $D$.

\subsection{Harmonic maps and actions of graphs}
The concept of harmonic maps was introduced in 
\cite{{baker-norine2},{urakawa1},{urakawa2}}. 

Let $G$, $G'$ be graphs. 
We say a function 
$\phi: V(G) \cup E(G) \rightarrow V(G')\cup  E(G')$ 
($\phi:G\rightarrow G'$ for short) 
is a {\it morphism} from $G$ to $G'$ 
if $\phi(V(G)) \subset V(G')$, and 
for any edge $e \in E(G)$ with endpoints $P$ and $Q$, 
one of the following holds: 
\begin{enumerate}
\item 
$\phi(e) \in E(G')$, and $\phi(P)$, $\phi(Q)$ 
are the endpoints of $\phi(e)$, 
\item 
$\phi(e) \in V(G')$, and $\phi(e) = \phi(P) = \phi(Q)$. 
\end{enumerate}
If $\phi(E(G)) \subset E(G')$, then we call $\phi$ a {\it homomorphism}, and 
a bijective homomorphism is called an {\it isomorphism}. 
An isomorphism from $G$ to $G$ is called an {\it automorphism}. 
We denote by $\Aut(G)$ the set of all automorphisms of $G$. 


A morphism $\phi: G \rightarrow G'$ is {\it harmonic} 
if for all $P \in V(G), Q \in V(G')$ with $Q = \phi(P)$ 
\[
|\{e \in E(G) \mid P \in e, \phi(e) = e'\}|
\]
is independent of the choice of $e' \in E(G')$, with $Q \in e'$.

A harmonic group action on a graph was defined in \cite{corry}. 
Let $\Gamma$ be a set of automorphisms of $G$. 
Then we define a quotient graph $G/\Gamma$ and a map 
\[
\phi_\Gamma: G\rightarrow G/\Gamma
\]
as
\[
\begin{cases}
V(G/\Gamma)&=V(G)/\Gamma,\\
E(G/\Gamma)&=E(G)/\Gamma\setminus \{\Gamma e\mid 
\mbox{$e$ has endpoints $P,Q$ and $\Gamma P=\Gamma Q$}\},\\
\phi_{\Gamma}(P)&=\Gamma P \mbox{ for all $P\in V(G)$}, 
\end{cases}
\]
and if $e\in E(G)$ with vertices $P$ and $Q$, 
then 
\[
\phi_{\Gamma}(e)=
\begin{cases}
\Gamma e \mbox{ if $\Gamma P\neq \Gamma Q$},\\ 
\phi_{\Gamma}(P)  \mbox{ if $\Gamma P= \Gamma Q$}. 
\end{cases}
\]
Let $\Gamma$ be a subgroup of $\Aut(G)$. 
Then $\Gamma$ acts {\it harmonically} on $G$ 
if for all subgroups $\Delta <\Gamma$ 
the quotient morphism 
$\phi_{\Delta} :G \rightarrow G/\Delta$ 
is harmonic. 

The following proposition is useful for checking whether $\Gamma$ is a harmonic action or not. 
\begin{proposition}[{\cite[Proposition 2.5]{corry},\cite[Proposition 2.7]{corry2}}]\label{prop:harm}
Suppose $\Gamma <\Aut(G)$ is a group of automorphisms of a graph $G$. 
Then $\Gamma$ acts {\it harmonically} if and only if 
for every vertex $P\in V(G)$ 
the stabilizer subgroup $\Gamma_P$ acts freely on 
$E(P(1))$, where 
$P(1)$ is the induced subgraph of $G$ on $P$. 
\end{proposition}

\subsection{$q$-reduced divisors}

In this subsection, we recall the concept of $q$-reduced divisors 
on finite graphs. 
We quote some results from \cite{corry-perkinson}. 

Let $D\in \Div(G)$ and $q\in V$. 
We say $D$ is {\it $q$-reduced} if the following two conditions are satisfied: 
\begin{enumerate}
\item 
$D(P)\geq 0$ for all $P\in V\setminus \{q\}$, 
\item
for all non-empty sets $S\subset V\setminus \{q\}$, 
there exists $P\in S$ such that 
\[
D(P)<\mbox{outdeg}_S(P), 
\]
where $\mbox{outdeg}_S(P)$ is the number of edges of $\overline{PQ}$ 
with $Q\not\in S$. 
\end{enumerate}
It is known that for $D\in \Div(G)$ and $q\in V$ 
there exists a unique $q$-reduced divisor $\widetilde{D}$ 
linearly equivalent to $D$ 
\cite[Theorem 3.7]{corry-perkinson}. 

We will now state and prove two technical 
lemmas that will be used later. 
\begin{lemma}\label{lem:pq}
Let $G$ be a 2-edge-connected graph, and 
for $P\in V(G)$, 
let $D=P$ be a divisor on $G$. 
Then there does not exist $Q\in V(G)\setminus \{P\}$ such that $P\sim Q$. 
\end{lemma}

\begin{proof}
Assume that $P\sim Q$. 
Let $q\in V(G)\setminus\{P,Q\}$. 
Then $P$ and $Q$ are $q$-reduced divisors. 
Indeed, let $S\subset V(G)\setminus \{q\}$. 
If 
there exists $R\in S\setminus\{P\}$ such that 
$\mbox{outdeg}_S(R)\geq 1$, 
then 
\[
0=D(R)<\mbox{outdeg}_S(R). 
\]
Therefore, we may assume that $\mbox{outdeg}_S(R)=0$ 
for all $R\in S\setminus\{P\}$. 
Since $G$ is 2-edge-connected, we have 
$\mbox{outdeg}_S(P)\geq 2$ and 
\[
1=D(P)<\mbox{outdeg}_S(P). 
\]
Hence, $P$ and $Q$ are $q$-reduced divisors. 
The uniqueness of the $q$-reduced divisor implies that 
$P=Q$. 
\end{proof}

\begin{lemma}\label{lem:D-P_1}
Let $G$ be a 2-edge-connected graph with 
\[
V(G)=\{P_1,\ldots,P_n\} 
\]
and 
\[
D=P_1+\cdots+P_n\in \Div(G). 
\]
Assume that 
$P_1$ is a Galois point with respect to $|D|$, and that 
$H$ is the group of order $n-1$ 
associated with $P_1$. 
Then for all $\sigma \in H$, 
\[
\sigma(P_1)=P_1,\ \sigma(D-P_1)=D-P_1. 
\]
\end{lemma}
\begin{proof}
Assume that there exists $\sigma \in H$ such that $\sigma(P_1) \ne P_1$. 
Let $E \in |D-P_1|$ be an effective divisor fixed by $H$, as in Definition \ref{Galois point} (3) (iii), and let $D-P_1=\Delta(f)+E$.  
Then 
\begin{eqnarray*} 
\sigma(D-P_1) &=&\sigma(D)-\sigma(P_1)=D-\sigma(P_1), \\ 
\sigma(\Delta(f)+E) &=& \sigma(\Delta(f))+\sigma(E)=\Delta(f\sigma^{-1})+E. 
\end{eqnarray*}
It follows that 
$$P_1-\sigma(P_1)=\Delta(f\sigma^{-1}-f), $$ 
namely, $P_1$ and $\sigma(P_1)$ are linearly equivalent. 
This is contradicts Lemma \ref{lem:pq}. 
\end{proof}

\section{Examples}\label{sec:ex}

\subsection{Complete graphs $K_n$}



For a complete graph $K_n$, the following holds. 

\begin{lemma} \label{lemma of ranks} 
Let $K_n$ be a complete graph with $|V(K_n)|=n \ge 4$, and let $D=P_1+\cdots+P_n$, where $\{P_1, \ldots, P_n\}=V(K_n)$. 
Then the following hold:  
\begin{itemize}
\item[(a)] $r(D)=2$, 
\item[(b)] for any $P \in V(G)$, $r(D-P)=1$, 
\item[(c)] for any $P, Q \in V(G)$ (including the case $P=Q$), $r(D-P-Q)=0$. 
\end{itemize}
\end{lemma}
\begin{proof}
\begin{enumerate}
\item [(a)]
Let 
\[
f(P)=
\begin{cases}
1\ \mbox{if $P=P_1$},\\
0\ \mbox{otherwise}. 
\end{cases}
\]
Then $\Delta(f)=(n-1)P_1-P_2-\cdots-P_n$, and we have 
\[D=P_1+\cdots+P_n\sim nP_1.\] 
Hence, we have $r(D)\geq 2$. 

Let $D'=(n-2)P_1-P_2$. We show $|D'|=\emptyset$. 
Indeed, assume the contrary: there exists $f\in M(G)$ 
such that 
\[
\Delta(f)+D'\geq 0. 
\]
\begin{enumerate}
\item[(i)] Assume that $f(P_1)\geq f(P_2)$. 
Without loss of generality, 
we may assume that 
\[
f(P_3)=\min\{f(P_k)\mid k\geq 3\}. 
\]
Then, we have 
\[
f(P_2)-f(P_k)\leq f(P_2)-f(P_3). 
\]
If $f(P_2)\leq f(P_3)$, then 
\begin{align*}
\Delta(f)&(P_2)=\sum_{i=1}^{n}(f(P_2)-f(P_i))\leq 0. 
\end{align*}
This contradicts $\Delta(f)(P_2)\geq 1$. 
Hence, we have \[f(P_2)> f(P_3), \]
and  
\begin{align*}
\Delta(f)&(P_3)=\sum_{i=1}^{n}(f(P_3)-f(P_i))\\
&\leq (f(P_3)-f(P_2))+(f(P_3)-f(P_1))< 0. 
\end{align*}
This contradicts $\Delta(f)(P_3)\geq 0$.

\item[(ii)] Assume that $f(P_1)<f(P_2)$. 
Without loss of generality, 
we may assume that 
\[
f(P_3)=\min\{f(P_k)\mid k\geq 3\}. 
\]
If $f(P_3)\leq f(P_1)$, then 
$f(P_3)<f(P_2)$. 
This contradicts $\Delta(f)(P_3)\geq 0$. 
Hence, we have \[f(P_3)>f(P_1).\] Then for all $k\geq 3$,
\[
f(P_k)\geq f(P_3)>f(P_1),
\]
and 
\[
\Delta(f)(P_1)=\sum_{i=2}^{n}(f(P_1)-f(P_i))\leq -(n-1). 
\]
This contradicts $\Delta(f)+D'\geq 0$. 
\end{enumerate}
\item [(b)]
It suffices to show that $r(D-P_1)=1$. 
This follows from 
$|(n-2)P_1-P_2)|=\emptyset$, which was proved in 
Lemma \ref{lemma of ranks} (a). 

\item [(c)]
It suffices to show that $r(D-P_{1}-P_2)=0$ and $r(D-2P_1)=0$. 
These follow from 
$|(n-2)P_1-P_2|=\emptyset$, which was proved in 
Lemma \ref{lemma of ranks} (a). 
\end{enumerate}
The proof is complete. 
\end{proof}

\begin{proposition}\label{lemma of Galois complete} 
Let $K_n$ be a complete graph with $|V(K_n)|=n \ge 3$, 
and let $D=P_1+\cdots+P_n$, where $\{P_1, \ldots, P_n\}=V(K_n)$. 
Then for any $P_i\in V(K_n)$, $P_i$ is a Galois point with 
respect to $|D|$.
\end{proposition}

\begin{proof}

We prove that $r(D)=2$ and conditions (1)--(3) in 
Definition \ref{Galois point} are satisfied for the vertex $P_1$.  
The other cases can be proved similarly. 
$r(D)=2$, conditions (1) and (2) 
are satisfied by Lemma \ref{lemma of ranks}. 
Let $\sigma \in {\rm Aut}(G)$ be an automorphism such that 
$$ \sigma(P_2)=P_3, \ldots, \sigma(P_i)=P_{i+1}, \ldots, \sigma(P_n)=P_2, $$
and 
$$ \sigma(\overline{P_iP_j})=\overline{\sigma(P_i)\sigma(P_j)} $$ 
for any $i, j$ with $2 \le i, j \le n$. 
The group of order $n-1$ generated by $\sigma$ is denoted by $H$. 
Then $H \cdot P_1=\{P_1\}$, and there does not exist another short orbit. 
Since the group $H$ acts on the set of edges with end point $P_1$ freely, 
and given Proposition \ref{prop:harm}, 
it follows that $H$ acts harmonically on $G$. 
It can be confirmed that 
\begin{align*}
\sigma(D-P_1)&=\sigma(P_2+\cdots+P_n)=P_2 + \cdots +P_n=D-P_1,\\
\sigma((n-1)P_1)&=(n-1)P_1, 
\end{align*}
for any $\sigma \in H$, and that 
condition (3) in Definition \ref{Galois point} is satisfied. 
\end{proof}
\begin{remark}
Since the genus $g$ of $K_n$ is equal to 
\[g=1-|V(K_n)|+|E(K_n)|=\frac{(n-1)(n-2)}{2}\] 
for $K_n$, we may consider $|P_1+\cdots+P_n|$ to be a complete linear system of an embedding to $\mathbb{P}^2$. 
\end{remark}
\subsection{Wheel graphs $W_n$}

For $n\geq 5$, let $W_n$ be the wheel graph with $n$ vertices: 
\[
\begin{cases}
V&=\{P_1,\ldots,P_n\},\\
E&=\{\overline{P_1P_i}\mid i\in \{2,\ldots,n\}\}\cup 
\{\overline{P_iP_{i+1}}\mid i\in \{2,\ldots,n-1\}\}
\cup\{\overline{P_nP_2}\}. 
\end{cases}
\]




\begin{figure}[h]
 \centering
\begin{tikzpicture}[scale=0.3]
\node (P1) [fill=white, draw, text=black, circle] at (0,0) {$P_1$};
\node (P2) [fill=white, draw, text=black, circle] at (5,0) {$P_2$};
\node (P3) [fill=white, draw, text=black, circle] at (0,5) {$P_3$};
\node (P4) [fill=white, draw, text=black, circle] at (-5,0) {$P_4$};
\node (P5) [fill=white, draw, text=black, circle] at (0,-5) {$P_5$};

\draw(P2)--(P3)--(P4)--(P5)--cycle;
\draw(P1)--(P2);
\draw(P1)--(P3);
\draw(P1)--(P4);
\draw(P1)--(P5);
\draw(P2)--(P5);

\end{tikzpicture}
\caption{$W_5$}\label{$W_5$}
\end{figure}

For a wheel graph $W_n$, the following holds. 

\begin{lemma} \label{lemma of ranks wheel} 
Let $W_n$ be a wheel graph with $|V(W_n)|=n \ge 5$, 
and let $D=P_1+\cdots+P_n$, where $\{P_1, \ldots, P_n\}=V(W_n)$. 
Then the following hold:  
\begin{itemize}
\item[(a)] $r(D)=2$, 
\item[(b)] for any $P\in V(G)$, $r(D-P)=1$, 
\item[(c)] for any $P,Q \in V(G)$ (including the case $P=Q$), $r(D-P-Q)=0$. 
\end{itemize}
\end{lemma}
\begin{proof}
\begin{enumerate}
\item [(a)]
Let 
\[
f_1(P)=
\begin{cases}
1\ \mbox{if $P=P_1$},\\
0\ \mbox{otherwise}. 
\end{cases}
\]
Then $\Delta(f_1)=(n-1)P_1-P_2-\cdots-P_n$, and we have 
\[D=P_1+\cdots+P_n\sim nP_1.\] 
Let 
\[
f_2(P)=
\begin{cases}
1\ \mbox{if $P=P_2$},\\
0\ \mbox{otherwise}. 
\end{cases}
\]
Then $\Delta(f_2)=-P_1+3P_2-P_3-P_{n}$, and we have 
\[D=P_1+\cdots+P_n\sim 4P_2+P_4+\cdots+P_{n-1}.\] 
Hence, we have $r(D)\geq 2$. 

Let $D'=(n-2)P_1-P_2$. We show $|D'|=\emptyset$. 
Indeed, assume the contrary: there exists $f\in M(G)$ 
such that 
\[
\Delta(f)+D'\geq 0. 
\]
Without loss of generality, 
we may assume that 
\[
f(P_i)=\min\{f(P_k)\mid k\geq 2\}. 
\]
\begin{enumerate}
\item[(i)] 
Assume that $f(P_1)\geq f(P_i)$. 
We have for all $j\in \{1,\ldots,n\}$, 
\[
f(P_i)-f(P_j)\leq 0. 
\]
If 
\begin{align*}
\Delta&(f)(P_i)=(f(P_i)-f(P_1))\\
&+(f(P_i)-f(P_{i-1}))+(f(P_i)-f(P_{i+1}))< 0, 
\end{align*}
where indices are considered modulo $n-1$, 
then 
this contradicts $\Delta(f)(P_i)\geq 0$. 
Hence, we may assume that 
\[
f(P_1)=f(P_i)=f(P_{i-1})=f(P_{i+1}). 
\]
Using $\Delta(f)(P_{i+1})\geq 0$, 
we have $f(P_{i+1})=f(P_{i+2})$. 
Inductively, 
we obtain 
for all $k\in \{2,\ldots,n\}$, 
$f(P_i)=f(P_k)$ and 
\begin{align*}
\Delta&(f)(P_2)=(f(P_2)-f(P_1))\\
&+(f(P_2)-f(P_{n}))+(f(P_2)-f(P_{3}))=0, 
\end{align*}
which contradicts $\Delta(f)(P_2)\geq 1$. 

\item[(ii)] Assume that $f(P_1)<f(P_i)$. 
\begin{align*}
\Delta(f)(P_1)=\sum_{j=1}^n (f(P_1)-f(P_j))\leq -(n-1)< 0. 
\end{align*}
This contradicts $\Delta(f)(P_1)\geq -(n-2)$. 

\end{enumerate}

\item [(b)]
It suffices to show that $r(D-P_1)=1$ and $r(D-P_2)=1$. 
This follows from 
$|(n-2)P_1-P_2|=\emptyset$, which was proved in 
Lemma \ref{lemma of ranks wheel} (a). 

\item [(c)]
It suffices to show that $r(D-P_{1}-P_2)=0$, $r(D-2P_1)=0$, 
and $r(D-2P_2)=0$. 
The first two cases follow from 
$|(n-2)P_1-P_2|=\emptyset$, which was proved in 
Lemma \ref{lemma of ranks wheel} (a). 

For the last case, 
let $D'=(n-1)P_1-2P_2$. We show $|D'|=\emptyset$. 
Indeed, assume the contrary: there exists $f\in M(G)$ 
such that 
\[
\Delta(f)+D'\geq 0. 
\]
Without loss of generality, 
we may assume that 
\[
f(P_i)=\min\{f(P_k)\mid k\geq 2\}. 
\]
\begin{enumerate}
\item[(i)] 
Assume that $f(P_1)\geq f(P_i)$. 
We have for all $j\in \{1,\ldots,n\}$, 
\[
f(P_i)-f(P_j)\leq 0. 
\]
If 
\begin{align*}
\Delta&(f)(P_i)=(f(P_i)-f(P_1))\\
&+(f(P_i)-f(P_{i-1}))+(f(P_i)-f(P_{i+1}))< 0, 
\end{align*}
where indices are considered modulo $n-1$, then 
this contradicts $\Delta(f)(P_i)\geq 0$. 
Hence, we may assume that 
\[
f(P_1)=f(P_i)=f(P_{i-1})=f(P_{i+1}). 
\]
Using $\Delta(f)(P_{i+1})\geq 0$, 
we have $f(P_{i+1})=f(P_{i+2})$. 
Inductively, 
we obtain 
for all $k\in \{2,\ldots,n\}$, 
$f(P_i)=f(P_k)$ and 
\begin{align*}
\Delta&(f)(P_2)=(f(P_2)-f(P_1))\\
&+(f(P_2)-f(P_{n}))+(f(P_2)-f(P_{3}))=0, 
\end{align*}
which contradicts $\Delta(f)(P_2)\geq 2$. 

\item[(ii)] Assume that $f(P_1)<f(P_i)$. 
If $f(P_1)+1<f(P_i)$ or 
$f(P_j)<f(P_{j\pm 1})$ for some $j,j\pm 1\in\{2,\ldots,n\}$, 
where indices are considered modulo $n-1$, then 
\begin{align*}
\Delta(f)(P_1)=\sum_{j=1}^n (f(P_1)-f(P_j))\leq -n. 
\end{align*}
This contradicts $\Delta(f)(P_1)\geq -(n-1)$. 
Hence, we may assume that $f(P_1)+1=f(P_i)$ and 
$f(P_2)=\cdots =f(P_n)$. 
Then 
\begin{align*}
\Delta(f)&(P_2)=(f(P_2)-f(P_1))\\
&+(f(P_2)-f(P_3))+(f(P_2)-f(P_n))\leq 1. 
\end{align*}
This contradicts $\Delta(f)(P_2)\geq 2$. 
\end{enumerate}

\end{enumerate}
The proof is complete. 
\end{proof}

\begin{proposition}
Let $D=P_1+\cdots+P_n$ be a divisor on $W_n$. 
Then the following hold: 
\begin{itemize}
\item[(a)] $P_1$ is a Galois point with respect to $|D|$,  
\item[(b)] for $i\neq 1$, 
$P_i$ is not a Galois point with respect to $|D|$. 
\end{itemize} 
In particular, the number of Galois points is exactly one. 

\end{proposition}
\begin{proof}
We prove that $r(D)=2$ and conditions (1)--(3) 
in Definition \ref{Galois point} are satisfied for the vertex $P_1$.  
$r(D)=2$, conditions (1) and (2) are satisfied 
by Lemma \ref{lemma of ranks wheel}. 
Let $\sigma \in {\rm Aut}(G)$ be an automorphism such that 
$$ \sigma(P_2)=P_3, \ldots, \sigma(P_i)=P_{i+1}, \ldots, \sigma(P_n)=P_2, $$
and 
$$ \sigma(\overline{P_iP_{i+1}})=\overline{\sigma(P_i)\sigma(P_{i+1})} $$ 
for any $i$ with $2 \le i\le {n-1}$, and  
\[
\sigma(\overline{P_{n}P_{2}})=\overline{\sigma(P_{n})\sigma(P_{2})}, 
\]
and 
$$ \sigma(\overline{P_1P_{i}})=\overline{\sigma(P_1)\sigma(P_{i})} $$ 
for any $i$ with $2 \le i\le {n}$. 
The group of order $n-1$ generated by $\sigma$ is denoted by $H$. 
Then $H \cdot P_1=\{P_1\}$, and there does not exist another short orbit. 
Since the group $H$ acts on the set of edges with end point $P_1$ freely, 
and given Proposition \ref{prop:harm}, 
it follows that $H$ acts harmonically on $G$. 
It can be confirmed that 
\begin{align*}
\sigma(D-P_1)&=\sigma(P_2+\cdots+P_n)=P_2 + \cdots +P_n=D-P_1,\\
\sigma((n-1)P_1)&=(n-1)P_1, 
\end{align*}
for any $\sigma \in H$, and that 
condition (3) in Definition \ref{Galois point} is satisfied.


Assume that $P_i$ is a Galois point for some $i$ with $2 \le i \le n$. 
We can assume that $i=2$. 
Let $H$ be an associated group, and 
let $E_1, E_2 \in |D-P|$ be effective divisors, 
as in Definition \ref{Galois point} (3). 
Then 
\[{\rm supp}(E_1) \cap \{P_2, \ldots, P_n\} \ne \emptyset 
\mbox{ or }
{\rm supp}(E_2) \cap \{P_2, \ldots, P_n\} \ne \emptyset. 
\]
We can assume that 
\[{\rm supp}(E_1) \cap \{P_2, \ldots, P_n\} \ne \emptyset. 
\]
Let $P_i \in {\rm supp}(E_1) \cap \{P_2, \ldots, P_n\}$. 
Since $P_1$ is a unique vertex of degree $n-1 \ge 4$, it follows that $H$ fixes $P_1$. 
By Proposition \ref{prop:harm}, the orbit $H \cdot P_i$ coincides with the set $\{P_2, \ldots, P_n\}$. 
Since $H$ acts on $E_1$, it follows that ${\rm supp}(E_1) \supset H \cdot P_i$, namely, $E_1=P_2+\cdots+P_n=D-P_1$.   
Then, $D-P_2$ and $D-P_1$ are linearly equivalent. 
This implies that $P_2 \sim P_1$. 
This contradicts Lemma \ref{lem:pq}.

The proof is complete. 
\end{proof}

\section{Proof of Theorem \ref{graph}}\label{sec:proof}

\begin{proof}[Proof of Theorem \ref{graph}] 
Assume that condition (1) is satisfied, that is, $G=K_n$. 
According to Lemma \ref{lemma of ranks}, it follows that $r(D)=2$, and 
conditions (1) and (2) are satisfied for any vertex $P \in V(G)$. 
Furthermore, by Proposition \ref{lemma of Galois complete}, 
it follows that $P$ is a Galois point with respect to 
$|D|$ for any $P \in V(G)$. 


Assume that condition (2) is satisfied, namely, $P_1$ and $P_2$ are Galois points with respect to $|D|$. 
Let $H_1, H_2$ be the groups of order $n-1$ associated with $P_1, P_2$, respectively. 
By Lemma \ref{lem:D-P_1}, we have 
$\sigma(P_1)=P_1$ and $\sigma(D-P_1)=D-P_1$ for any $\sigma \in H_1$, 
namely, $H_1$ is a stabilizer subgroup of $P_1$. 
Since $G$ is connected and simple, and $H_1$ acts on the set of edges with end point $P_1$ freely, it follows that $P_2, \ldots, P_n$ are connected to $P_1$. 
This implies that the group $H_1$ acts on the set $V(G)\setminus\{P_1\}$ transitively. 
Similarly, the group $H_2$ fixes $P_2$ and acts on the set $V(G) \setminus \{P_2\}$ transitively. 
Then, the action of the group $\langle H_1, H_2 \rangle$ on $V(G)$ is doubly transitive. 
This implies that $P_i$ is connected to $P_j$ for any $i, j$ with $i \ne j$, namely, that $G=K_n$. 
\end{proof} 

Focusing on the number of Galois points, we have the following. 

\begin{corollary} 
Let $G$ be a 2-edge-connected graph with $n=|V(G)|$, and let $D:=P_1+\cdots+P_n$. 
Assume that $r(D)=2$. 
Then the number of Galois points with respect to $|D|$ is $0$, $1$, or $n$. 
Furthermore, the number is equal to $n$ if and only if $G=K_n$.  
\end{corollary} 

\begin{remark}
Let $D=P_1+\cdots+P_n$. 
There exist examples of graphs admitting no 
Galois points with respect to $|D|$. 
It can be confirmed that $r(D)=2$ and there does not exist a Galois point with respect to $|D|$ for the graph $G=(V,E)$: 
\[
\begin{cases}
V=\{P_1,P_2,P_3,P_4\},\\
E=\{\overline{P_1P_2},\overline{P_2P_3},\overline{P_3P_4},\overline{P_4P_1},\overline{P_1P_3}\}. 
\end{cases}
\]
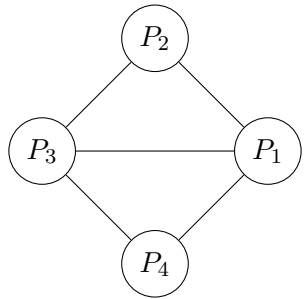
\begin{figure}[h]
 \centering
\begin{tikzpicture}[scale=0.3]
\node (P1) [fill=white, draw, text=black, circle] at (5,0) {$P_1$};
\node (P2) [fill=white, draw, text=black, circle] at (0,5) {$P_2$};
\node (P3) [fill=white, draw, text=black, circle] at (-5,0) {$P_3$};
\node (P4) [fill=white, draw, text=black, circle] at (0,-5) {$P_4$};

\draw(P1)--(P2)--(P3)--(P4)--cycle;
\draw(P1)--(P3);
\draw(P1)--(P4);

\end{tikzpicture}
\end{figure}

We recall the Riemann--Roch theorem for graphs \cite{baker-norine}: 
\[
r(D)-r(K-D)=\deg(D)+1-g, 
\]
where $g=1-|V(G)|+|E(G)|$ and 
\[
K=\sum_{P\in V(G)}(\deg(P)-2)P. 
\]
Then we have 
\[
r(D)-r(-P_2-P_4)=4+1-2 \Leftrightarrow r(D)=2. 
\]

On the other hand, for any $i\in \{1,\ldots,4\}$, 
$P_i$ is not a Galois point. 
We only show the case $i=1$ and the other cases can 
be proved similarly. 
Indeed, if not, then 
there exists a subgroup $H\subset \Aut(G)$ of order $|H|=3$ 
such that $H_{P_1}$ acts freely on $V(G)\setminus \{P_1\}$ 
by Proposition \ref{prop:harm} and Lemma \ref{lem:D-P_1}. 
This contradicts $2=\deg(P_2)\neq \deg(P_3)=3$.
\end{remark}

\section{Concluding remarks}\label{sec:rem}

\begin{remark}
For any graph $G$ and any vertex $P \in V(G)$ of degree $k$, there exists a natural harmonic morphism $\varphi_P$ of degree $k$ from $G$ to a tree (see \cite[Example 3.2]{baker-norine2}). 
Then we can define an {\it intrinsic Galois point} $P$ (tentatively) as follows: there exists a subgroup $H \subset {\rm Aut}(G)$ of order $k$ such that $H$ acts harmonically on $G$ and $\varphi_P \circ \sigma=\varphi_P$ for any $\sigma \in H$. 
This notion may be close to that for a ``Galois-Weierstrass point,'' introduced by Morrison and Pinkham (see \cite{komeda-takahashi}). 
Conditions (1) and (2) in Theorem \ref{graph} are equivalent to the following condition. 
\begin{itemize}
\item[(3)] There exist different vertices $P, Q \in V(G)$ of degree $n-1$ such that $P, Q$ are intrinsic Galois points. 
\end{itemize}
\end{remark}

\begin{remark}
In algebraic geometry, the divisor $C.L$ coming from the intersection $C \cap L$ of a smooth plane curve $C \subset \mathbb{P}^2$ and a projective line $L \subset \mathbb{P}^2$ is a typical example of a divisor $D$ for which $r(D)=2$ (for example, see \cite[Theorem 1.4.9]{namba}). 
Assume that the characteristic of the ground field is zero. 
For smooth plane curves, the number of Galois points contained in $C$ is $0$, $1$, or $4$ (see \cite{yoshihara}). 
In the last case, $\deg (C)=4$, and Galois points $P_1, P_2, P_3$, and $P_4$ are contained in some line $L$. 
The divisor $C.L$ arising from $C \cap L$ coincides with $P_1+P_2+P_3+P_4$.  
\end{remark}

\section*{Acknowledgments}
The authors were supported 
by JSPS KAKENHI (22K03223, 22K03277).


\begin{thebibliography}{20} 
\bibitem{baker-norine} M.~Baker and S.~Norine, Riemann--Roch and Abel--Jacobi theory on a finite graph, {\sl Adv.~Math.} {\bf 215} (2007), 766--788. 

\bibitem{baker-norine2} M.~Baker and S.~Norine, Harmonic morphisms and hyperelliptic graphs, {\sl IMRN} {\bf 15} (2009), 2914--2955. 






\bibitem{corry} S.~Corry, Genus bounds for harmonic group actions on finite graphs, {\sl IMRN} {\bf 19} (2011), 4515--4533. 

\bibitem{corry2} S.~Corry, 
Harmonic Galois theory for finite graphs, 
Adv.~Stud.~Pure Math.~{\bf 63} (2012), 121--140. 

\bibitem{corry-perkinson} 
S.~Corry and D.~Perkinson, 
{\sl Divisors and sandpiles. An introduction to chip-firing.} 
American Mathematical Society, Providence, RI, 2018.

\bibitem{fukasawa} S.~Fukasawa, Galois points for a plane curve in arbitrary characteristic, {\sl Geom.~Dedicata} {\bf 139} (2009), 211--218.  

\bibitem{hartshorne} R.~Hartshorne, {\sl Algebraic Geometry}, 
Graduate Texts in Mathematics {\bf 52}, Springer-Verlag, New York, 1977. 

\bibitem{komeda-takahashi} 
J.~Komeda and T.~Takahashi, 
Relating Galois points to weak 
Galois Weierstrass points through double coverings of curves, 
{\sl J.~Korean Math.~Soc.} {\bf 54} (2017), 69--86. 

\bibitem{miura-yoshihara} K.~Miura and H.~Yoshihara, 
Field theory for function fields of plane quartic curves, 
{\sl J.~Algebra} {\bf 226} (2000), 283--294. 

\bibitem{namba} 
M.~Namba, 
{\sl Geometry of Projective Algebraic Curves}, 
Marcel Dekker, New York, 1984. 

\bibitem{stichtenoth} H.~Stichtenoth, {\sl Algebraic Function Fields and Codes}, Graduate Texts in Mathematics {\bf 254}, Springer-Verlag, Berlin Heidelberg, 2009.  


\bibitem{urakawa1} 
H.~Urakawa, 
A Discrete Analogue of the Harmonic Morphism. 
In Harmonic Morphisms, Harmonic Maps, and Related Topics, 97--108. 
Chapman \& Hall/CRC Research Notes in Mathematics 413. 
Boca Raton, FL: Chapman \& Hall/CRC, 2000. 

\bibitem{urakawa2} 
H.~Urakawa, 
A discrete analogue of the harmonic morphism and 
Green kernel comparison theorems. 
{\sl Glasg.~Math.~J.} {\bf 42}, no.~3 (2000), 319--334.

\bibitem{yoshihara} H.~Yoshihara, 
Function field theory of plane curves by dual curves, 
{\sl J.~Algebra} {\bf 239} (2001), 340--355. 

\bibitem{open} H.~Yoshihara and S.~Fukasawa, 
List of problems, 
\url{https://sites.google.com/sci.kj.yamagata-u.ac.jp/fukasawa-lab/open-questions-english}
\end{thebibliography}
\end{document}